\documentclass[12pt]{amsart}

\usepackage[english]{babel}
 \usepackage[utf8]{inputenc}
\usepackage[T1]{fontenc}
\usepackage{amsmath, amssymb}

\usepackage[a4paper,top=3cm,bottom=2cm,left=3cm,right=3cm,marginparwidth=1.75cm]{geometry}

\usepackage{amsthm}
\usepackage{amsfonts}
\usepackage{amssymb}
\usepackage{amsmath}
\usepackage{listings} 
\usepackage{xcolor} 
\usepackage{upquote} 
\usepackage{faktor}
\usepackage{graphicx}
\usepackage{csquotes}
\usepackage{enumitem}
\usepackage[colorinlistoftodos]{todonotes}
\usepackage{url}
\usepackage[style=alphabetic, hyperref,bibencoding=latin1, giveninits=true,isbn=false, backend=bibtex]{biblatex}
\usepackage{etoc}

\newtheorem{theorem}{Theorem}
[section]

\newtheorem{corollary}[theorem]{Corollary}
\newtheorem{definition}[theorem]{Definition}
\newtheorem{lemma}[theorem]{Lemma}
\newtheorem{proposition}[theorem]{Proposition}

\theoremstyle{remark}
\newtheorem{remark}[theorem]{Remark}

\def\CC{\mathbb{C}}
\def\FF{\mathbb{F}}

\def\NN{\mathbb{N}}
\def\QQ{\mathbb{Q}}
\def\RR{\mathbb{R}}
\def\ZZ{\mathbb{Z}}
\def\vv{\mathbf{v}}
\def\xx{\mathbf{x}}
\def\yy{\mathbf{y}}

\def\calC{\mathcal{C}}
\def\calD{\mathcal{D}}
\def\calF{\mathcal{F}}
\def\calH{\mathcal{H}}
\def\calM{\mathcal{M}}
\def\calo{\mathcal{O}}

\newcommand{\nrd}{\mathrm{nrd}\,}

\newcommand{\alg}[3]{ \left(\frac{#1,#2}{#3}\right)}

\newcommand{\trd}{\mathrm{trd}}

\usepackage{listings}
\usepackage{xcolor}

\definecolor{codegreen}{rgb}{0,0.6,0}
\definecolor{codegray}{rgb}{0.5,0.5,0.5}
\definecolor{codepurple}{rgb}{0.58,0,0.82}
\definecolor{backcolour}{rgb}{0.95,0.95,0.92}

\lstdefinestyle{sagestyle}{
    backgroundcolor=\color{backcolour},   
    commentstyle=\color{codegreen},
    keywordstyle=\color{magenta},
    numberstyle=\tiny\color{codegray},
    stringstyle=\color{codepurple},
    basicstyle=\ttfamily\footnotesize,
    breakatwhitespace=false,         
    breaklines=true,                 
    captionpos=b,                    
    keepspaces=true,                 
    numbers=none,                    
    numbersep=5pt,                  
    showspaces=false,                
    showstringspaces=false,
    showtabs=false,                  
    tabsize=2
}

\usepackage{amsmath, amssymb}
\usepackage[utf8]{inputenc}
\newcommand{\notalea}[1]{\todo[color=violet!40]{#1}}

\DeclareMathOperator{\QCFF}{QCFF}

\date{}
\title{Quaternionic $p$-adic continued fractions}

\author{Laura Capuano}
\address{Dipartimento di Matematica e Fisica, Università degli Studi Roma Tre, L.go S. L. Murialdo 1, 00146 Roma, Italy}
\email{laura.capuano@uniroma3.it}

\author{Marzio Mula}
\address{Dipartimento di Matematica, Università di Trento, Via Sommarive 14, I-38123 Povo (Trento), Italy}
\email{marzio.mula@unitn.it}

\author{Lea Terracini}
\address{Dipartimento di Informatica, Università di Torino, Corso Svizzera 185 - 10149 Torino, Italy}
\email{lea.terracini@unito.it}

\date{}

\begin{document}

\maketitle
\begin{abstract}We develop a theory of $p$-adic continued fractions for a quaternion algebra $B$ over $\QQ$ ramified at a rational prime $p$. Many properties holding in the commutative case can be proven also in this setting. In particular, we focus our attention on the characterization of elements having a finite continued fraction expansion. By means of a suitable notion of quaternionic height, we prove a criterion for finiteness. Furthermore, we draw some consequences about the solutions of a family of quadratic polynomial equations with coefficients in $B$.
\end{abstract}
\section{Introduction}\label{sect:Introductionzero}
Given a division algebra $B$ and a list $[a_0,\ldots, a_n]$ of non-zero elements in $B$,  the simple continued fraction
\begin{align*}
    a_0 + \cfrac{1}{a_1+ \cfrac{1}{\ddots + \cfrac{1}{a_n}}} 
\end{align*}
represents an element $\alpha\in B$. When $B$ is also provided with a topology, under some  hypotheses of completeness and convergence, infinite continued fractions can be considered too. In this general setting, several problems arise, concerning:
\begin{itemize}
\item representability of elements in the topological completion of $B$;
\item uniqueness of the representation;
\item existence of an algorithm to compute the continued fraction expansion of any representable element;
\item convergence and quality of the approximation;
\item characterizations of those elements in $B$ which are represented by finite continued fractions;
\item characterizations of those elements in $B$ which have a periodic continued fraction.
\end{itemize}
\par The questions above have been extensively studied in the ``classical'' case, i.e.\ when $B$ is the real field endowed with the Euclidean topology. The first contributions date back to the works of Wallis, Euler and Lagrange~\cite{brezinski2012}. More recent works, starting from~\cite{Ruban1970} and~\cite{Browkin1978}, deal with the case where $B$ is a number field endowed with a non-archimedean absolute value.
\par In this paper, we address the case where $B$ is a quaternion algebra over $\QQ$  ramified at a rational rime $p$, and the topology is the $p$-adic one. The main properties of such setting are summarized in Section~\ref{sect:Introduction}. While the idea of defining continued fractions over quaternions was already suggested by Hamilton~\cite{Hamilton1852}, a standard terminology for the $p$-adic case is still missing. Therefore, in analogy with the case of number fields~\cite{CapuanoMurruTerracini2021}, in Section~\ref{sec:contFrac} we define the notion of \emph{quaternionic type} associated to an order $R$ in $B$. Namely, this is a quadruple $\tau=(B,R,p,s)$ where $B$ is a quaternion algebra, $R$ is an order in $B$, $p$ a prime $\ge 3$ and $s$ a \lq\lq $p$-adic floor function\rq\rq\ taking values in $ R[ \frac  1 p ] $. 
Each quaternionic type gives rise to an algorithm that computes the continued fraction expansion of every element in the $p$-adic completion of $B$. We show that the convergents of such continued fractions enjoy many of the properties holding also in the commutative case. Given a quaternionic type $\tau$, we approach the problem of characterizing all the elements of $B$ whose continued fraction expansion is finite. To this end, in Section~\ref{sect:height} we introduce a notion of \emph{quaternionic height} for the pair $(B,R)$, which is a particular instance of a height function associated to an adelic norm as defined in \cite{Talamanca2004}. In Section~\ref{sect:criteria}, we use Northcott property to prove Theorem \ref{teo:muForInfiniteCF}, which provides a criterion to establish the finiteness of the continued fraction algorithm in some cases. Adopting a similar terminology as in~\cite{CapuanoMurruTerracini2021}, we say that a type $\tau$ satisfies the \emph{Quaternionic Continued Fraction Finiteness} ($\QCFF$) if every $\alpha \in B$ has a finite continued fraction expansion of type $\tau$.
Section \ref{sect:Bpq} deals with indefinite quaternion algebras of discriminant $pq$: in this case, we give an explicit expression for a maximal order, and we construct some concrete examples of quaternionic types for which the $\QCFF$ does not hold. Finally, in Section \ref{sect:polinomi} we draw some consequence of Theorem \ref{teo:muForInfiniteCF} regarding the existence of solutions of some quadratic polynomial equations with coefficients in~$B$.

\section*{Acknowledgements}
We thank Valerio Talamanca for useful conversations regarding his work on heights associated to an adelic norm. The three authors are part of the GNSAGA group of INdAM.

\section{Some generalities on quaternion algebras over $\QQ$}\label{sect:Introduction}
We shall denote by $\mathcal{M}$ the set of rational places, and by $\mathcal{M}^0$ the subset of non-archimedean places. Therefore,every $v\in \mathcal{M}^0$ corresponds to a rational prime $q$ and is associated to an absolute value $|\cdot |_q$ normalized in such a way that $|q|_q= \frac 1 q$. If $v=\infty $, then  $|\cdot |_\infty$ is the usual absolute value.\par
We refer to \cite{Vignéras1980,voight:quaternion} for the basic definitions and properties of quaternion algebras and orders. In what follows:
\begin{itemize}
    \item $B$ is a quaternion algebra over $\QQ$ of discriminant $\Delta>1$;
    \item $p$ is an odd prime dividing $\Delta$;
    \item $R$ is an order in $B$.
\end{itemize}
If $x\in B$, we shall denote  by $\nrd(x)$ and $\trd(x)$ the reduced norm and trace of $x$, respectively.\\
For every rational place $v \in \mathcal{M}$, we put
$$B_v=B\otimes_\QQ \QQ_v$$
and, if $v$ is a non-archimedean place,
$$R_v=R\otimes_\ZZ \ZZ_v.$$
Therefore, $R_v$ is an order in $B_v$.
\par Let us consider an odd prime $q$ such that $B$ is \emph{ramified} at $q$, i.e.\
$B_q$
is a division algebra.  Following~\cite[§13.3]{voight:quaternion}, $B_q$ contains a unique maximal order, that is
$$R_q^{max}=\{\alpha\in B_q\ |\ v_q(\mathrm{nrd}(\alpha))\geq 0\}.$$
Notice that $R_q^{max}$ is a local ring; moreover, the $q$-adic valuation $v_q \colon \mathbb{Q}_q \rightarrow \mathbb{R} \cup \{\infty\}$ can be extended  to a valuation
 on $B_q$, defined as
\[w_q(\alpha)=\frac{v_q(\nrd(\alpha))}{2}.\]
The map $w_q$ is a discrete valuation (see~\cite[Lemma 13.3.2]{voight:quaternion});  this ensures that one can define $q$-adic non-archimedean absolute value over $B_q$:
\begin{align}
\label{eqn:pAdicNormDef}
    |\alpha|_q=\left( \frac{1}{q} \right)^{w_q(\alpha)}\qquad \text{for each $\alpha \in B_q$}.
\end{align}

The following result shows that $B_q$ is unique up to isomorphism.
\begin{theorem}
\label{thm:Bnu}
    Up to $\mathbb{Q}_q$-algebra isomorphism, we have
    \begin{equation*}
        B_q \cong \alg{\mathbb{Q}_{q^2}}{q}{\mathbb{Q}_q}=\mathbb{Q}_{q^2} \oplus \mathbb{Q}_{q^2} j,
    \end{equation*}
    where $j^2=q$ and $\mathbb{Q}_{q^2}$ denotes the unique quadratic unramified\,\footnote{i.e. $q$ is also a generator for the maximal ideal of the valuation ring of $\mathbb{Q}_{q^2}$.} (separable) extension of $\mathbb{Q}_q$.
\end{theorem}

\begin{proof}
See~\cite[Thm.\,13.3.11]{voight:quaternion}.
\end{proof}
\begin{remark}
\label{rem:globalUniformizer}
Notice that the standard generators of $B_q$ can be chosen in $B$. To see this, write $\mathbb{Q}_{q^2}\cong \mathbb{Q}_q[X]/f(X)$ for some irreducible quadratic polynomial $f(X)=X^2+a_1X+a_0 \in \mathbb{Q}_q[X]$. As a consequence of Krasner's Lemma~\cite[Cor.\,13.2.9]{voight:quaternion}, there exists a constant $\delta>0$ such that any polynomial $g(x)=X^2+b_1X+b_0$ with $|b_i-a_i|_q<\delta$ satisfies 
\[\mathbb{Q}_q[X]/f(X) \cong \mathbb{Q}_q[X]/g(X).\]
In particular, one can choose $g(X) \in \mathbb{Q}[X]$, and set the generator $i$ as a root of $g(X)$. Up to ``completing the square'' ($q$ is an odd prime), we can assume $i^2 \in \mathbb{Q}$.
\end{remark}

\section{Quaternionic continued fractions}\label{sec:contFrac}

In this section we introduce a suitable notion of $p$-adic quaternionic continued fractions. To do this, we first need to define an analogue of the usual floor function with good properties which guarantee the convergence of the algorithm. We start with the following definition.

\begin{definition} A \emph{$p$-adic floor function} for the pair $(B,R)$ is a function $s \colon B_p \rightarrow B$ such that
\begin{itemize}
    \item $|\alpha - s(\alpha)|_{p}<1$ for every $\alpha \in B_{p}$;
    \item $s(\alpha)\in R_q$ for every prime $q\not= p$;
    \item $s(0)=0$;
    \item $s(\alpha) = s(\beta)$ if $|\alpha - \beta|_{p} <1$.
\end{itemize}
\end{definition}

It is easy to see that, given a pair $(B,R)$, the $p$-adic floor function is not unique, as the following result shows.

\begin{theorem}
    \label{thm:ffexistence} There exist infinitely many $p$-adic floor functions for the pair $(B,R)$.
    \end{theorem}
    \begin{proof}
    Let $e_1, e_2, e_3, e_4$ be a basis of $R$ over $\ZZ$. Let $\alpha\in B_p$ and write $\alpha=\sum_{i=1}^4 \alpha_ie_i$ with $\alpha_i\in\QQ_p$. Fix $0< \epsilon<1$. By strong approximation, for $i=1,..,4$, there exists $\beta_i\in\QQ$ such that
    \begin{itemize}
        \item $|\beta_i-\alpha_i|_p<\epsilon$;
        \item $|\beta_i|_q\leq 1$ if $p\not=q$.
    \end{itemize} 
    Put $\beta=\sum_{i=1}^4\beta_ie_i$. Then $\beta\in R_q$ for every prime $q\not=p$. Since $e_i\in R_q$ for $i=1,\ldots, 4$ we also  have
       $$ |\alpha-\beta|_p\leq \max_i\{|\alpha_i-\beta_i|_p|e_i|_p\}\leq \max_i\{|\alpha_i-\beta_i|_p\}<1.$$
       Then, we can define $s(\alpha)=\beta$.
       By letting $\epsilon$ tend to $0$, we obtain infinitely many $p$-adic floor functions for $(B,R)$.
\end{proof}
\par We now have the essentials tools for extending the classical continued fraction algorithm to our setting. We keep the notation of Section \ref{sect:Introduction}: let $s$ be a $p$-adic floor function for the pair $(B,R)$, and $\alpha_0$ be an element in $B_{p}$.\\
 The following recursive algorithm computes the continued fraction expansion of $\alpha_0$:
\begin{equation}
\label{eqn:contFracAlg}
     \begin{cases}
a_n = s(\alpha_n),\\
\alpha_{n+1}=(\alpha_n - a_n)^{-1} \qquad \text{if $\alpha_n - a_n \neq 0$},

\end{cases}
\end{equation}
and the algorithm stops if $\alpha_n = a_n$. Thus, the \emph{continued fraction expansion of $\alpha_0$} is the (possibly infinite) sequence $[a_0, a_1, \dots]$. The $a_i$-s are called \emph{partial quotients}, while the $\alpha_i$-s are called \emph{complete quotients}.
\begin{proposition}
\label{prop:an>1}
    For every $n\geq 1$, we have $|a_n|_{p}>1$. 
\end{proposition}
\begin{proof} By definition of $a_n$,
     \begin{align*}
         |a_n|_{p} =|s(\alpha_n)|_{p}=\left|s\bigl((\alpha_{n-1} - a_{n-1})^{-1}\bigr)\right|_{p}= \left|s\bigl((\alpha_{n-1} - s(\alpha_{n-1}))^{-1}\bigr)\right|_{p}.
    \end{align*}
    Since $s$ is a floor function, $|\alpha_n|_{p}=|(\alpha_{n-1} - s(\alpha_{n-1}))^{-1}|_{p}>1$. Moreover, since $|\cdot|_p$ is non-archimedean, the ultrametric inequality holds:
    \[ \underbrace{|\alpha_n|_{p}}_{>1} \leq \max \bigl\{ |\underbrace{s(\alpha_n)-\alpha_n|_{p}}_{<1}, |s(\alpha_n)|_{p}\bigr\}.\]
    Therefore, $|a_n|_{p}$ must be $>1$.
\end{proof}

We call a \emph{quaternionic type} any quadruple $\tau=(B, R, p, s)$ such that
\begin{itemize}
    \item $B$ is a division quaternion algebra over $\mathbb{Q}$;
    \item $R$ is an order of $B$;
    \item $p$ is  a prime number such that$B$ is ramified at $p$;
    \item $s$ is a $p$-adic floor function for the pair $(B,R)$.
\end{itemize}

\subsubsection*{Special types}
\label{subsubsec:specialtypes} Assume that $R$ contains an element $\pi$ such that $\nrd(\pi)=\pm p$. Then $\pi$ is a uniformizer of $R_p$. Since $R$ is dense in $R_p$, there is an isomorphism
$R/\pi R\simeq R_p/\pi R_p\simeq \FF_{p^2}$. Let $\mathcal{C}\subseteq R$ be a complete set of representatives for the quotient $R/\pi R$. Then, every $\alpha\in B_p$ can be expressed uniquely as a Laurent series  $\alpha=\sum_{k=-n}^\infty c_k\pi^k$, where $c_k\in\mathcal{C}$ for every $k$. It is possible to define a $p$-adic floor function for the pair $(B,R)$ by
$$s(\alpha)=\sum_{k=-n}^0c_k\pi^k\in B. $$
In analogy with \cite[\S 3.2]{CapuanoMurruTerracini2021}, we shall denote the types $\tau=(B,R,p,s)$ obtained in this way by $\tau=(B,R,\pi,\calC)$, and we will usually call them \emph{special types}.\\
For example, when $B$ is indefinite and $R$ is an Eichler order, then it is proven~\cites[Cor.\,17.8.5]{voight:quaternion}[Cor.\,5.9]{Vignéras1980} that every ideal in a maximal order is principal, so that there exists at least an element $\pi\in R$ of reduced norm $\pm p$. In fact there are infinitely many such elements, since the group of elements of reduced norm 1 in $R$ is infinite, as one can easily deduce from \cite[Thm.\,4.1.1]{Vignéras1980}.

\bigskip
\par If a quaternionic type $\tau$ is fixed, a \emph{(quaternionic) continued fraction of type $\tau$} is any sequence $[a_0,a_1\dots ]$ of elements of $\mathrm{Im}(s)$ such that  $|a_i|_{p}>1$ for each $i\neq 0$. A periodic sequence of the form $[a_0, a_1,..., a_k,a_0,a_1\dots]$ is usually denoted by $[\overline{a_0,\dots,a_k}]$.

\par Similar definitions as for the classical case can be adopted in the (non-commutative) quaternionic setting: for any continued fraction $[a_0, a_1, \dots]$, we define the sequences
\begin{align*}
    A_{-1}=1,&& A_0=a_0, && A_n=A_{n-1}a_n + A_{n-2} &&\text{for $n \geq 1$},\\
    B_{-1}=0, && B_0=1, && B_n=B_{n-1}a_n +B_{n-2} &&\text{for $n \geq 1$},
    \end {align*}
and the matrices
    \begin{align*}
    \mathcal{A}_n &=\begin{pmatrix}
    a_n & 1 \\ 
    1 & 0
    \end{pmatrix}
   && \text{for $n \geq 0$},\\
     \mathcal{B}_n &=\begin{pmatrix}
    A_n & A_{n-1} \\ 
    B_n & B_{n-1}
    \end{pmatrix}
    && \text{for $n \geq 0$}.
\end{align*}
\begin{proposition}
    For each $n\geq 0$,
    \begin{equation}
    \label{eqn:Bn}
        \mathcal{B}_n=\mathcal{A}_0\cdots \mathcal{A}_{n}.
    \end{equation}
\end{proposition}
\begin{proof}
    We first observe
    \begin{align*}
         \mathcal{B}_n &=\begin{pmatrix}
    A_n & A_{n-1} \\ 
    B_n & B_{n-1}
    \end{pmatrix}\\
    &=\begin{pmatrix}
    A_{n-1}a_n + A_{n-2} & A_{n-1}\\
    B_{n-1}a_n + B_{n-2} & B_{n-1}
    \end{pmatrix}\\
    &=\begin{pmatrix}
    A_{n-1} & A_{n-2} \\ 
    B_{n-1} & B_{n-2}
    \end{pmatrix}\cdot
    \begin{pmatrix}
    a_{n} & 1\\ 1 & 0
    \end{pmatrix}\\
    &=\mathcal{B}_{n-1} \mathcal{A}_n
    \end{align*}
    for any $n>0$. Therefore, \eqref{eqn:Bn} immediately follows by induction.
\end{proof}
As usual, the \emph{$n$-th convergent} is
\begin{align*}
    Q_n=a_0 + \cfrac{1}{a_1+ \cfrac{1}{\ddots + \cfrac{1}{a_n}}} \quad \text{for $n \geq 0$},
\end{align*}
where the notation $1/\alpha$ stands for $\alpha^{-1}$.\\
As in the commutative case, $Q_n$ can be expressed in terms of $A_n$ and $B_n$.
\begin{proposition}
For $n \geq 0$
    \begin{align}
    \label{eqn:Qn}
    Q_n= {A_n}(B_n)^{-1}.
\end{align}
\end{proposition}
\begin{proof}
     We prove the statement for \emph{any} sequence $\{a_n\}_{n \in \mathbb{N}}$ of elements in $B$, no matter if such sequence is a continued fraction or not. For $n=0$, the equality is trivial.
    Assume that the statement holds for any sequence of length $n-1$. Given $a_0,\dots,a_n$, define a new sequence
    \begin{equation*}
        \begin{cases}
            \tilde{a}_i=a_i \quad\text{for $0\le i<n-1$},\\
            \tilde{a}_{n-1}=a_{n-1} + \frac{1}{a_n},
        \end{cases}
    \end{equation*}
    and denote by $\tilde{Q}_i$ the corresponding partial convergents. Then
    \begin{align*}
        Q_n &= \tilde{Q}_{n-1}\\
        &=\tilde{A}_{n-1}(\tilde{B}_{n-1})^{-1}\\
        &=(A_{n-2}\tilde{a}_{n-1}+A_{n-3})(B_{n-2}\tilde{a}_{n-1} + B_{n-3})^{-1},
    \end{align*}
    where the second equality follows by inductive hypothesis. On the other hand,
    \begin{align*}
        {A_n}(B_n)^{-1}&=(A_{n-1}a_n + A_{n-2})(B_{n-1}a_n + B_{n-2})^{-1}\\
        &=\bigl((A_{n-2}a_{n-1} + A_{n-3})a_n + A_{n-2} \bigr) \bigl((B_{n-2}a_{n-1} + B_{n-3})a_n + B_{n-2} \bigr)^{-1}\\
        &= \bigl(A_{n-2}(a_{n-1}a_n +1) + A_{n-3}a_n\bigr) \bigl(B_{n-2}(a_{n-1}a_n +1) + B_{n-3}a_n\bigr)^{-1}\\
        &= \bigl(A_{n-2}(a_{n-1}a_n +1) + A_{n-3}a_n\bigr)\cdot a_n^{-1} \cdot a_n \cdot \bigl(B_{n-2}(a_{n-1}a_n +1) + B_{n-3}a_n\bigr)^{-1}\\
        &=(A_{n-2}\tilde{a}_{n-1}+A_{n-3})(B_{n-2}\tilde{a}_{n-1} + B_{n-3})^{-1}.
    \end{align*}
\end{proof}

\begin{proposition}
    For any $n>0$, the following equalities hold:
    \begin{align}
        \label{eqn:BnPadicNorm}
        |B_n|_p&= \prod_{j=1}^n |a_j|_p,\\
        \label{eqn:QnDiff}
        |Q_n-Q_{n-1}|_p&=\frac{1}{|B_n|_p|B_{n-1}|_p}.    \end{align}
\end{proposition}
\begin{proof}
     We first prove \eqref{eqn:BnPadicNorm} by induction: the case $n=1$ holds trivially since $B_1=a_1$. For $n=2$ we have $|B_2|_p=|a_1a_2+1|_p= |a_2a_1|_p$, where the second equality is granted by the fact that $|a_1|_p, |a_2|_p>1$ and the ultrametric inequality. Similarly, let us assume inductively that~\eqref{eqn:BnPadicNorm} holds for $n-1$ and $n-2$; then, we have 
     $$|B_n|_p= \max \{|a_nB_{n-1}|_p,|B_{n-2}|_p\}= |a_n|_p  \cdot\prod_{j=1}^{n-1} |a_j|_p= \prod_{j=1}^n |a_j|_p.$$
     As for~\eqref{eqn:QnDiff}, if $n=1$ we have
     \[|Q_1-Q_0|_p= \frac{1}{|a_1|_p}=\frac{1}{|B_1|_p|B_{0}|_p}.\]
     If $n\geq 2$, then we have:
     \begin{align*}
         Q_n - Q_{n-1}&=A_nB_n^{-1}-A_{n-1}B_{n-1}^{-1}\\
         &= (A_{n-1}a_n + A_{n-2})( B_{n-1}a_n + B_{n-2})^{-1} - A_{n-1}B_{n-1}^{-1}\\
         &= \bigl(A_{n-1}a_n + A_{n-2} - A_{n-1} B_{n-1}^{-1}(B_{n-1}a_n + B_{n-2})\bigr)\cdot ( B_{n-1}a_n + B_{n-2})^{-1}\\
         &=\bigl(A_{n-1}a_n + A_{n-2} - A_{n-1}a_n - A_{n-1}B_{n-1}^{-1}B_{n-2}\bigr)\cdot ( B_{n-1}a_n + B_{n-2})^{-1}\\
         &=(A_{n-2}B_{n-2}^{-1}-A_{n-1}B_{n-1}^{-1})B_{n-2}\cdot ( B_{n-1}a_n + B_{n-2})^{-1}\\
         &=(Q_{n-2}-Q_{n-1})\cdot (B_{n-1}a_nB_{n-2}^{-1}+1)^{-1}.
     \end{align*}
     Let us consider the second factor of the latter equality: by~\eqref{eqn:BnPadicNorm} and the ultrametric inequality we have
     \begin{align*}
         |B_{n-1}a_nB_{n-2}^{-1}+1|_p= |B_{n-1}a_nB_{n-2}^{-1}|_p= |a_{n-1}a_n|_p>1.
     \end{align*}
    Thus we may conclude
    \begin{align*}
        |Q_n - Q_{n-1}|_p=\frac{1}{|a_1|_p}\cdot\frac{1}{\prod_{i=2}^n |a_{i-1}a_i|_p}=\frac{1}{|B_n|_p|B_{n-1}|_p}. && \qedhere
    \end{align*}
\end{proof}
This implies easily the following result.

\begin{corollary}
    \label{cor:QnIsCauchy}
    The sequence $\{Q_n\}_{n \in \mathbb{N}}$ is convergent with respect to the $p$-adic topology.
\end{corollary}

\medskip
The previous corollary ensures that a sequence $\{Q_n\}_{n \in \mathbb{N}}$ is convergent. However, we still need to check that, given a sequence of convergents defined by applying Algorithm~\eqref{eqn:contFracAlg} to some element, the $p$-adic limit  coincides with the element itself. To prove this, some preliminary results are needed.

\begin{lemma}
    \label{lem:alphaVSa}
    For each $n \geq 0$,
    \[|\alpha_n|_p= |a_n|_p.\]
\end{lemma}
\begin{proof}
    The case $n=0$ is immediate. Fix any $n\geq 1$. 
  Algorithm \eqref{eqn:contFracAlg} yields
     \[\alpha_n=a_n + \frac{1}{\alpha_{n+1}},\]
     and we already showed that (see the proof of Proposition \ref{prop:an>1})
    \[|\alpha_{n+1}|_p>1.\]
     Therefore, the thesis follows from the ultrametric inequality.
\end{proof}

\begin{lemma}\label{lem:quattroquindici}
    For each $n \geq 1 $,
    \begin{equation}
    \label{eqn:alpha0Qn}
        \alpha_0=(A_{n}\alpha_{n+1}+A_{n-1})(B_{n}\alpha_{n+1}+B_{n-1})^{-1}.
    \end{equation}
\end{lemma}
\begin{proof}
     For every $n \geq 0$, let us substitute $a_n$ by $\alpha_n$ in the expression of $Q_n$ and denote by $\tilde{Q}_n$ the resulting element, i.e.\
    \begin{align*}
    \begin{cases}
    \tilde{Q}_0=\alpha_0,\\
    \tilde{Q}_n=a_0 + \cfrac{1}{a_1+ \cfrac{1}{\ddots + \cfrac{1}{\alpha_n}}} \quad \text{for $n \geq 1$}.
    \end{cases}
\end{align*}
Algorithm \eqref{eqn:contFracAlg} yields
     \[\alpha_n=a_n + \frac{1}{\alpha_{n+1}},\]
     so that $\tilde{Q}_n=\tilde{Q}_{n-1}$ and therefore $\tilde{Q}_n=\alpha_0$ for each $n \geq 1$.
     On the other hand, by \eqref{eqn:Qn},
     \[\tilde{Q}_{n+1}=(A_{n}\alpha_{n+1}+A_{n-1})(B_{n}\alpha_{n+1}+B_{n-1})^{-1}\]
     for $n \geq 1$.
     
\end{proof}
\begin{proposition}
\label{prop:CFconvergence}
    If the continued fraction expansion of $\alpha_0 \in B_p$ is infinite, it converges $p$-adically to $\alpha_0$.
\end{proposition}
\begin{proof}
  By \eqref{eqn:pAdicNormDef} we have
  \begin{align*}
      \alpha_0-A_n(B_n)^{-1} &= (A_n\alpha_{n+1}+ A_{n-1})(B_n\alpha_{n+1}+ B_{n-1})^{-1}- A_n(B_n)^{-1}\\
      &=\bigl(A_n\alpha_{n+1}+ A_{n-1}-A_n(\alpha_{n+1}+B_n^{-1}B_{n-1})\bigr) (B_n \alpha_{n+1} + B_{n-1})^{-1}\\
      &= (A_{n-1}-A_n B_n^{-1}B_{n-1})(B_n\alpha_{n+1}+B_{n-1})^{-1}\\
      &=(A_{n-1}B_{n-1}^{-1}-A_nB_n^{-1})B_{n-1}(B_n\alpha_{n+1}+B_{n-1})^{-1}\\
      &=(Q_{n-1}-Q_n)(B_n\alpha_{n+1}B_{n-1}^{-1}+1)^{-1}.
  \end{align*}
  The first factor $p$-adically converges to $0$ thanks to Corollary \ref{cor:QnIsCauchy}, while the second factor converges to $0$ by \eqref{eqn:BnPadicNorm} and Lemma \ref{lem:alphaVSa}.
\end{proof}

\begin{proposition}
\label{prop:uniquenessCF}
 Let $\{a_n\}_{n \in \mathbb{N}}$ be a sequence of elements in $B$ such that, for each $n \in \mathbb{N}$,
     \begin{itemize}
        \item $|a_n|_p>1$,
        \item there exists a maximal order $R\subset B$ such that $a_n\in R_q$ for every prime $q \neq p$,
        \item if $|a_i-a_j|_p<1$, then $a_i=a_j$.
    \end{itemize}
    Then there exists a floor function $s$ such that $[a_0,a_1,\dots]$ is a continued fraction of type $\tau=(B,R,p,s)$.\\
    Moreover, let $\alpha_0 \in B_p$ be the $p$-adic limit of $[a_0,a_1,\dots]$. Then, $a_0, a_1, \dots$ are exactly the partial quotients of the continued fraction expansion of $\alpha_0$.
\end{proposition}
\begin{proof}
 For each $\alpha \in B_p$, let us define $s(\alpha)=a_i$ if $|\alpha- a_i|_p<1$ for some $i\in \mathbb{N}$; otherwise, let us construct $s$ as in Theorem~\ref{thm:ffexistence}. It is immediate to check that the resulting function is a floor function.
 
 For each $n \in \mathbb{N}$, let $Q_n$ be $n$-th convergent of $[a_0,a_1,\dots]$. Since we are assuming that $\{Q_n\}_{n \in \mathbb{N}}$ converges $p$-adically to $\alpha_0$, there exists $m\geq 1$ such that $|\alpha_0- Q_m|_p<1$. Then, ultrametric inequality and~\eqref{eqn:QnDiff} yield
 \begin{equation*}
     |\alpha_0-Q_{m-1}|_p\leq \max\{|\alpha-Q_m|_p, |Q_m-Q_{m-1}|_p\}<1.
 \end{equation*}
 Thus, after iterating the above argument $m$ times, we conclude $|\alpha_0-a_0|_p<1$ since $Q_0=a_0$. In particular, $s(\alpha_0)=a_0$ by definition of $s$. Setting $\alpha_{n}=(\alpha_{n-1}-a_{n-1})^{-1}$ for each $n\geq 1$, one can inductively check that $\alpha_n$ is the $p$-adic limit of $[a_n,a_{n+1},\dots]$. Therefore, $|\alpha_n-a_n|_p<1$ by the same argument. This proves that $s(\alpha_n)=a_n$ and $\{a_n\}_{n \in \mathbb{N}}$ (resp. $\{\alpha_n\}_{n \in \mathbb{N}}$) are the partial (resp.\ complete) quotients of the continued fraction expansion of $\alpha_0$, as wanted.
\end{proof}

For any $n \geq -1$, let us define
  \[V_n=A_n-\alpha B_n.\]
Then, one can prove the following useful equalities.

  \begin{proposition}\label{prop:propertiesVn}
      For each $n \geq 1$, the following relations hold:
      \begin{enumerate}[label=\roman*)]
          \item $V_n= V_{n-1} a_n+ V_{n-2}$.
          \item $V_{n-1}\alpha_n + V_{n-2}=0$.
          \item $V_{n-1}=(-1)^{n}\alpha_{1}^{-1}\cdots \alpha_{n}^{-1}$.
          \item $|V_{n-1}|_p=\prod_{j=1}^{n}\frac{1}{|a_j|_p}$.
      \end{enumerate}
  \end{proposition}
  
  \begin{proof} \
    \begin{enumerate}[label=\roman*)]
        \item Let us prove it by induction on $n$. The case $n=1$ is an easy verification. Assume that the claim is true for every $m<n$; then by definition we have
        \begin{align*}
            V_n&=A_n - \alpha B_n \\
            &= A_{n-1}a_n + A_{n-2} - \alpha B_{n-1}a_n - \alpha B_{n-2}\\
            &= (\underbrace{A_{n-1}-\alpha B_{n-1}}_{V_{n-1}}) a_n + \underbrace{A_{n-2}-\alpha B_{n-2}}_{V_{n-2}},
        \end{align*}
        as wanted.
        \item We can mimic the proof of Lemma \ref{lem:quattroquindici}: for every $n \geq 0$, let us substitute $a_n$ by $\alpha_n$ in the expression of $V_n$ (resp. $Q_n$) and denote by $\tilde{V}_n$ (resp. $\tilde{Q}_n$) the resulting element. We have already observed $\tilde{Q}_n=\alpha$ for each $n \geq 1$. Thus,
        \[\tilde{V}_n = (\alpha - \tilde{Q}_n)B_n=0.\]
        On the other hand, (i) ensures
        \[\tilde{V}_n = V_{n-1}\alpha_n + V_{n-2},\]
        proving the claim.
        \item From (ii) we get
        \[\alpha_n=-V_{n-1}^{-1}V_{n-2},\] so that
        \begin{align*}
            \alpha_n  \cdots \alpha_1= (-1)^n V_{n-1}^{-1}
        \end{align*}
        since $V_{-1}=1$. 
        \item Follows immediately from (iii) and Lemma \ref{lem:alphaVSa}.
    \end{enumerate}
  \end{proof}
  
\section{Quaternionic heights}
\label{sect:height}
In this section we define a suitable notion of quaternionic height, using the work of Talamanca \cite{Talamanca2004} in which the author defines a height function associated to an adelic norm.

We begin by revising some basic definitions and properties related to such heights. For this, we follow \cite{Talamanca2004}.

\subsection{Local norms over a vector space}
Let $q$ be a rational prime.
 Let $V$ be a finite-dimensional vector space over $\QQ_q$.  A subset $\Omega\subseteq V$ is a \emph{$\ZZ_q$-lattice} if it is a compact open $\ZZ_q$-module.\\
 Every $\ZZ_q$-lattice $\Omega\subseteq V$ defines a \emph{norm} $N_\Omega$ on $V$ by
 $$N_\Omega(\mathbf{v})= \inf_{\lambda\in \QQ_q, \lambda\vv\in\Omega} |\lambda|_q^{-1}.$$
 $N_\Omega$ is an ultrametric norm on $V$, that is
 \begin{itemize}
     \item $N_\Omega(\lambda\vv)=|\lambda |_q N_\Omega(\vv)$ for all $\vv\in V$;
     \item $N_\Omega(\vv_1+\vv_2)\leq \max\{N_\Omega(\vv_1),N_\Omega(\vv_2)\}$ for all $\vv_1,\vv_2\in V$.
 \end{itemize}

\subsection{Adelic norms on vectors spaces over global fields}
Let now $V$ be a $\QQ$-vector space.
Let $M\subseteq V$ be a lattice, that is a finitely generated subgroup containing a basis  of $V$ over $\QQ$. \\
Put $V_q=V\otimes_\QQ \QQ_q$, for every $q\in \calM$, and $M_q=M\otimes_{\ZZ} \ZZ_q$, for every $v\in \calM^0_K$.\\
A family of norms $\calF=\{N_v: V_v\to\RR\ , v\in\calM_K\}$ is said to be an \emph{adelic norm} on $V$ if 
\begin{itemize}
    \item every $N_v$ is a norm on $K_v$, ultrametric if $v\in\calM^0(K)$;
    \item there exists an $\calo_K$-lattice $M\subseteq V$ such that $N_v=N_{M_v}$ for all but finitely many $v\in\calM^0(K)$.
    \end{itemize}
\begin{remark}
The last condition implies that if $\xx\in V$ then $N_v(\xx)=1$ for all but finitely many $v$.
\end{remark}

\subsection{Height function associated to an adelic norm}\label{ssect:heightfunction}
Given an adelic norm $\calF$, the \emph{height function} on $V$ associated to $\calF$ is
$$\calH_\calF(\xx)=\prod_{v\in\calM(K)} N_v(\xx)^{d_v}.$$

Notice that this is well-defined since the product is finite by the remark above.

We will denote by $$\calH(V)=\{\calH_\calF\ |\ \calF\hbox{ is an adelic norm on } V\}$$
the set of height functions associated to adelic norms.
\medskip

\noindent The following properties are proven in \cite[Prop.\,1.1]{Talamanca2004}.

    \begin{itemize} 
    \item By the product formula,
$$\calH_\calF(\lambda \xx)=\calH_\calF(\xx)\hbox{ if }\lambda\in \QQ^\times,$$
so that $\calH_\calF$ descends on a function on $\mathbb{P}(V)$, i.e., the projective space associated to $V$.
\item If $\calH_1,\calH_2\in \calH(V)$, there exists a constant $C=C(\calH_1,\calH_2)>1$ such that, for all $\xx\in V$,
$$\frac 1 C \calH_1(\xx)\leq \calH_2(\xx)\leq C\calH_1(\xx).$$
\item {\bf Northcott property:}
For all $C>0$ the set 
$$\{[\xx ]\in\mathbb{P}(V)\ |\ \calH(\xx )< C\}$$
is finite.
\end{itemize}

\subsection{The quaternionic case}\label{ssect:quatcase}
Now let $B$ be quaternion algebra over $\QQ$, and $R$ be an order in $B$. Then, $R$ is by definition a lattice in $B$. 

We consider the adelic norm $\calF=\{N_v, v\in\calM\}$ on $B$, where
\begin{itemize}
\item $N_v=N_{R_v}$ if $v$ is non-archimedean and $B$ is unramified at $v$;
\item $N_v=|\cdot |_v$ (induced by the discrete valuation) if $v$ is non-archimedean and $B$ is ramified at $v$;
    \item $N_v$ is the operator norm on $M_2(\RR)$ or $M_2(\CC)$ if $v$ is archimedean.
    \end{itemize}

For these norms, it is straightforward to verify that the following multiplicative properties hold.
    
\begin{proposition}
  For every  $v$ the norm $N_v$ is submultiplicative, that is
$$N_v(\xx\cdot \yy)\leq N_v(\xx)N_v(\yy),\quad\forall \xx,\yy\in B.$$
Moreover, if $v$ is non-archimedean and ramified, then $N_v$ is multiplicative 
$$N_{v}(\xx\cdot \yy)= N_{v}(\xx)N_{v}(\yy),\quad\forall \xx,\yy\in B.$$ 
\end{proposition}

\section{A Criterion for finiteness}
\label{sect:criteria}
Let $\tau=(B,R,p,s)$ be a quaternionic type. In this section, we prove a sufficient criterion to decide whether a type satisfies the QCFF. 
To do this, we first need to prove an elementary result regarding real linear recurrence sequences.

For any $x\in\CC$, let us define
$$\theta(x)=\frac 1 2(|x|_\infty +\sqrt{|x|_\infty^2+4});$$
then, we have the following inequality: 
$$|x|_\infty\leq \theta(x)\leq |x|_\infty+1, $$
and the map $\theta$ is a bijection from $[0,+\infty)$ to $[1,+\infty)$ whose inverse is given by $y\mapsto y-\frac 1 y$. 

\begin{lemma}\label{lem:matrixnorm} Let $(c_n)_{n\ge 1}$ be any sequence of real numbers $\geq 0$ and let $(t_n)_{n \ge -1}$ be a sequence of real numbers  $\geq 0$ satisfying, for every $n\ge 1$ the inequality:
$$t_n\leq c_nt_{n-1}+t_{n-2}.$$
Then, there exists $c>0$ such that, for every $n\geq 0$,
$$\max\{t_n,t_{n-1}\}\leq c\cdot \prod_{j=1}^n \theta(c_j) .$$
\end{lemma}
\begin{proof}
For any complex matrix $M$, let us consider the operator norm 
\begin{align*}
||M|| &= \sup_{{\bf v}\not={\bf 0}} \frac {||M{\bf v}||}{||{\bf v}||},
\end{align*}
where $||{\bf v}||$ denotes the Euclidean norm of a complex vector. The following facts are well known (see for example \cite[§5]{Horn2013}):
\begin{itemize}
    \item $||M_1\cdot M_2||\leq ||M_1||\cdot ||M_2||$;
    \item  $||M||=\sqrt{|\gamma|_\infty}$, where $\gamma$ is the dominant eigenvalue of $M\cdot M^*$ (here $M^*$ denotes the transpose conjugate of $M$).
\end{itemize}
In particular, we have that, for every $a\in\CC$,
\begin{equation*}
\left |\left | \begin{pmatrix} a & 1\\ 1 & 0\end{pmatrix} \right |\right |=\theta(a).\end{equation*} 
 Let $\mathcal{M}_n=\begin{pmatrix} c_n & 1\\ 1 &0 \end{pmatrix}$; 
then, for every $n\geq 1$ we have
$$ \left|\left | \begin{pmatrix} t_n\\t_{n-1}\end{pmatrix} \right|\right |\leq  \left|\left |\mathcal{M}_n\begin{pmatrix} t_{n-1}\\t_{n-2}\end{pmatrix}\right|\right |\leq \left|\left |\mathcal{M}_n \right |\right| \left|\left | \begin{pmatrix} t_{n-1}\\t_{n-2}\end{pmatrix}\right|\right|=\theta(c_n)\left|\left | \begin{pmatrix} t_{n-1}\\t_{n-2}\end{pmatrix}\right|\right|,$$
so that 
$$ \max\{|t_n|_\infty, |t_{n-1}|_\infty\}\leq \left|\left | \begin{pmatrix} t_n\\t_{n-1}\end{pmatrix} \right|\right |\leq c\cdot \prod_{j=1}^n \theta(c_j),$$
with $c=\left|\left | \begin{pmatrix} t_0\\t_{-1}\end{pmatrix} \right|\right |$, as wanted.
\end{proof} 

\medskip
Consider the adelic norm $\calF$ defined in Section \ref{ssect:quatcase} and
let $\calH$ be the height function associated to $\calF$ as in Section \ref{ssect:heightfunction}.
Let $\alpha\in B$ and put $V_n=A_n-\alpha B_n$; then, by Proposition \ref{prop:propertiesVn} $iv)$,
\begin{itemize}
    \item $N_{p}(V_n)=\left |\prod_{j=1}^n \frac 1 {a_j}\right|_{p}= \prod_{j=1}^n \frac 1 {| a_j|_{p}}=(\frac 1 p)^{\sum_{j=1}^n w_p(a_j)}$, where
$w_p$ is the discrete valuation on $B_{p}$;
\item for a non-archimedean $q\not=p$,
$$N_q(V_n)\leq \max\{N_q(A_n),N_q(\alpha)N_q(B_n)\}\leq \max\{N_q(\alpha),1\};$$
\item for the archimedean $v=\infty$,
\begin{align*}
    N_\infty(V_n)&=N_\infty(V_{n-1}a_n+V_{n-2})\\
    &\leq N_\infty(V_{n-1})N_\infty(a_n)+N_\infty(V_{n-2})&&&\text{(by triangular inequality}\\ & &&& \text{and submultiplicativity)}\\
    &\leq C_\infty(\alpha)\prod_{j=1}^n \theta(N_\infty(a_j)) &&& \text{(by  Lemma \ref{lem:matrixnorm}).}
\end{align*}
\end{itemize}

These estimates for the norms allow us to prove a criterion to characterize the elements of the quaternion algebra having finite continued fraction expansion.

Following the terminology introduced in \cite{MVV2020} in the real case and in \cite{CapuanoMurruTerracini2021}, we shall say that a quaternionic type $\tau$ satisfies the \emph{Quaternionic Continued Fraction Finiteness ($\QCFF$)}  property if every $\alpha \in B$ has a finite expansion of type $\tau$. We say that a pair $(B,R)$ has the \emph{$p$-adic $\QCFF$ property} if there exists a quaternionic type $(B,R,p,s)$ enjoying the $\QCFF$ property. We have the following result.

\begin{theorem}
\label{teo:muForInfiniteCF}
    Let $\tau=(B,R,p,s)$ be a quaternionic type and let $\alpha\in B$ having an infinite continued fraction expansion
    $[a_0,a_1,a_2\ldots]$ of type $\tau$.
   Assume that there exists an eventual upper bound $\mu_\alpha$ for the sequence
   $$\left \{\frac{\theta(N_\infty(a_n))}{|a_n|_{p}}\right \}_{n\in\NN};$$
   then, $\mu_\alpha \geq  1$.
    \end{theorem}
    \begin{proof}
        We denote by
        $$ C(\alpha) = C_\infty(\alpha)\cdot \prod_{ q\not=p }\max\{N_q(\alpha),1\};$$
then
\begin{align*}
      \calH(V_n) &= \prod_{v\in\calM} N_v(V_n)\\
      &\leq C_\infty(\alpha) \prod_{j=1}^n \theta(N_\infty(a_j))\cdot \prod_{ q\not=p }\max\{N_q(\alpha),1\}\cdot\prod_{j=1}^n \frac 1 {| a_j|_{p}}\\
     & \leq  C(\alpha)\cdot \mu_\alpha^n. \end{align*}  
     If $\mu_\alpha<1$, then $\calH(V_n)\to 0$. By the Northcott property, we have that $V_n=0$ for $n\gg 0$ and the continued fraction is finite, giving a contradiction.
\end{proof}

This implies the following criterion for the QCFF property.

\begin{corollary}
Let $\tau=(B,R,p,s)$ be a quaternionic type. Define
\[
\mu=\sup\left\{\frac{\theta(N_\infty(a))}{|a|_{p}}\ |\  a\in s(B), |a|_{p}>1\right\}.
\]
If $\mu<1$, then $\tau$ satisfies the $QCFF$ property.
\end{corollary}

\subsection{Bounded types} A type $\tau=(B,R,p,s)$ is said to be \emph{bounded} if there exists a real number $C>0$ such that $N_\infty(s(B))=\sup \{N_\infty(a)\ |\ a\in s(B) \}<C$. 

\begin{proposition}
For every triple $(B,R,p)$ there exists a floor function $s$ such that the type $(B,R,p,s)$ is bounded.
\end{proposition}
\begin{proof}
Let 
$$P=\{x\in R\ |\ |x|_p<1\}.$$
Then, $P$ is a lattice in $B_\infty$, so that there is a bounded fundamental domain $\calD\subseteq B_\infty$ for the quotient $B_\infty/P$.
We construct a $p$-adic floor function $s$ for $B$ as follows. Let $\pi$ be a uniformizer in $R_p$, and let us consider a non trivial coset $\alpha+\pi R_p \subseteq  B_p$; by strong approximation, it contains an element $\alpha'\in B$ such that $\alpha'\in R_q$ for every rational prime $q\not=p$. Possibly translating  $\alpha' $ by a suitable element of $P$, we find a $\beta\in R[\frac 1 p]$ such that $\beta\in \calD$ and $\alpha'\equiv \beta \pmod{P}$.  Then, for every $\gamma\in \alpha+\pi R_p$ we put  $s(\gamma)=\beta$ and $\tau=(B,R,p,s)$.
\end{proof}

\begin{theorem} Assume that $\tau$ is a bounded type; then, there exists a positive integer $K$ such that every infinite continued fraction $[a_0,a_1,\dots]$ of type $\tau$ either represents an element $\alpha\in B_p\setminus B$ or the set
\(\{i \mid  |a_i|_p\leq {\sqrt{p}^K}\}\)
is infinite.
\end{theorem}
\begin{proof}
Since $\tau$ is bounded, there exists a real number $C>0$ such that $N_\infty(a_n)<C$ for every $n\ge 0$.
Choose $K>0$ such that
$C<\sqrt{p}^{K}.$ Assume that $\alpha\in B_p$ has an infinite expansion of type $\tau$ in which only finitely many partial quotients have absolute value $\leq  {\sqrt{p}^K}$; 
we want to show that $\alpha\not\in B$. By hypothesis 
$N_\infty(a_n)\leq C$, so that, for $n\gg 0$,
$$\frac{\theta(N_\infty(a_n))}{|a_n|_p
}\leq \frac{C+1}{\sqrt{p}^{K+1}}\leq\frac{\sqrt{p}^K+1}{\sqrt{p}^{K+1}}<1.$$
Therefore $\mu_\alpha<1$ and we can apply Theorem \ref{teo:muForInfiniteCF} to get the conclusion.
\end{proof}

\section{Construction of a $p$-adic type when $\Delta=pq$}\label{sect:Bpq}
Fix a prime $p\geq 3$. We construct an indefinite division quaternion algebra $B$ ramified at $p$, with the further requirement that $B$ is ramified only at $p$ and at another place $q\neq \infty$ (we recall that the number of ramified places of a quaternion algebra must be even~\cite[Thm.\,14.6.1]{voight:quaternion}). To ensure this, it is enough~\cite[Lem.\,1.21]{Montserrat2004} to choose a prime $q\equiv 1 \mod 4$  such that $\left(\frac{p}{q}\right)=-1$ and set $B=\alg{q}{p}{\mathbb{Q}}$. If $p \equiv 3 \mod 8$, also $q=2$ can be chosen.

\medskip
\par Consider the \emph{standard order} $R'=\mathbb{Z} + \mathbb{Z}i + \mathbb{Z}j + \mathbb{Z} ij$. One can prove~\cite[Ex.\,15.2.10]{voight:quaternion} that the reduced discriminant of this order is $2pq$, while the discriminant of $B$ (i.e.\ the product of all ramified places) is $pq$ by construction. Therefore $R'$ is not maximal~\cite[Thm.\,15.5.5]{voight:quaternion}. However, it is contained in a unique maximal order $R$ by~\cite[Prop.\,1.32.iii]{Montserrat2004}. By~\cite[Prop.\,1.60]{Montserrat2004},
 $R$ has an explicit expression of the form:
\[R=\begin{cases}
\mathbb{Z} + \mathbb{Z}i +\mathbb{Z}j+ \mathbb{Z}\frac{1+i+j+ij}{2}\qquad &\text{if $q=2$,}\\
\mathbb{Z} + \mathbb{Z}i +\mathbb{Z}\frac{1+j}{2}+ \mathbb{Z}\frac{i+ij}{2}\qquad& \text{otherwise.}
\end{cases}\]


\par We mimic the classic construction of continued fractions in $\mathbb{Q}_p$ given by Browkin~\cite{Browkin1978}. A natural choice is to use a special type as described in Section~\ref{sec:contFrac}. 

First, we choose a set $\mathcal{C}$ of representatives for $R/jR$:
\begin{equation*}
    \mathcal{C}=\left\{a+ b i \mid a,b \in \left\{0,\pm 1, \dots, \pm \frac{1}{2}(p-1)\right\}\right\}.
\end{equation*}
Thus, given $\alpha \in B$, we can define a floor function as follows:
\begin{itemize}
    \item write $\alpha$ as the series
    \[\alpha=\sum_{\ell=r}^\infty \alpha_\ell j^\ell\]
    where $\alpha_\ell\in \mathcal{C}$ for each $\ell$.
    \item set
    \begin{equation*}
        s(\alpha)=\begin{cases}
    0 \qquad&\text{if $r>0$,}\\
    \sum_{\ell=r}^0 \alpha_\ell j^\ell \qquad&\text{if $r\leq 0$.}
    \end{cases}
    \end{equation*}
\end{itemize}
\par It is clear that the elements of $\mathbb{Q}$, seen as a subset of $B$, enjoy a finite continued fraction expansion with respect to the special type $\tau=(B,R,j,\mathcal{C})$ constructed above. In fact, their expansion coincides with the classic Browkin continued fraction expansion considered in~\cite{Browkin1978}. However, there exist elements whose continued fraction expansion is infinite.
\begin{theorem}
\label{thm:QCFFfail}
  The type $\tau=(B,R,j,\mathcal{C})$ constructed above does not satisfy the QCFF property.
\end{theorem}
\begin{proof}
Suppose that the $n$-th complete quotient of some continued fraction has the following form: \[\alpha_n=\frac{k_1}{k_2p^r}\left(i+\frac{1}{p}ij\right),\]
where $k_1, k_2$ are coprime integers not divided by $p$, $k_2 \notin\{-1,1\}$ and $r\geq0$. We write the Bézout's identity for $p^{r+1}$ and $k_2$, i.e.
\[
vp^{r+1} + wk_2=k_1,
\]
choosing the integers $v,w$ in such a way that $w\in \{-(p^{r+1}-1)/2,\dots, (p^{r+1}-1)/2\}$. 
Therefore, one can check that
\[
a_n=s(\alpha_n)=\frac w {p^r}\left(i+\frac{1}{p}ij\right).
\]
Moreover, since $k_2 \notin\{-1,1\}$, we have that $v \neq 0$, so the continued fraction does not terminate and the next complete quotient is
\begin{align}
\label{eqn:alphaexample}
    \alpha_{n+1}=(\alpha_n-a_n)^{-1}=\frac{k_2p^r}{vp^{r+1}}\cdot \left(i+\frac{1}{p}ij\right)^{-1}=\frac{k_2}{q(p-1)v}\left(i+\frac{1}{p}ij\right).
\end{align}
We claim that the continued fraction expansion of \[\alpha_0=\frac{1}{q}\left(i+\frac{1}{p}ij\right)\]
is infinite. Indeed, by~\eqref{eqn:alphaexample},
\begin{align*}
    \alpha_{1}=\frac{1}{(p-1)v_1}\left(i+\frac{1}{p}ij\right),
\end{align*}
where $v_1$ satisfies Bézout's identity
\begin{equation}
    \label{eqn:bez1}
    v_1p+w_1q=1
\end{equation} for some integer $w_1$. After writing $v_1=v_1'p^{r_1}$ with $p \nmid v_1'$ and $r_1\geq0$, the next complete quotient can be computed  using~\eqref{eqn:alphaexample}:
\begin{align*}
    \alpha_{2}=\frac{v_1'}{qv_2}\left(i+\frac{1}{p}ij\right),
\end{align*}
where $v_2$ satisfies the equality 
\begin{equation}
    \label{eqn:bez2}
    v_2p^{r_1+1}+w_2(p-1)v_1'=1
\end{equation}for some integer $w_2$. Notice that $q$ cannot divide $v_1'$: otherwise, the contradiction  $1 \equiv 0 \mod q$ would follow from~\eqref{eqn:bez1}. Moreover, $v_1'$ and $v_2$ have no common factors because of~\eqref{eqn:bez2}. Similarly, from~\eqref{eqn:alphaexample} we get
\begin{align*}
    \alpha_{3}=\frac{v_2'}{(p-1)v_3}\left(i+\frac{1}{p}ij\right)
\end{align*}
where $v_2=v_2'p^{r_2}$ with $p \nmid v_2'$  and $r_2\geq0$, 
\[v_3p^{r_2+1}+w_3 q v'_2=v_1'\]
for some integer $w_3$, and $p-1$ does not divide $v_2'$ because $v_2p+w_2(p-1)v_1'=1 \not \equiv 0 \mod (p-1)$. Moreover, $v_3$ and $v_2'$ have no common factors (otherwise, such factor would be also a common factor of $v_1'$ and $v_2$). The form of each $\alpha_n$ can now be derived by induction.
Namely, first define the sequence  $\{v_n\}_{n\in \NN}$ as follows: $v_{-1}=v_0=1$, and, for each $n\geq1$,
$v_n$ is the unique integer such that
\begin{equation*}
 v_n p^{r_{n-2}+r_{n-1}+1} + w_nCv_{n-1}p^{r_{n-2}}=p^{r_{n-1}}v_{n-2}\qquad \text{with}\quad C=
    \begin{cases}
   p-1\quad &\text{if $n$ is even},\\
   q\quad &\text{if $n$ is odd},
    \end{cases}
\end{equation*}
for some $w_n\in \{-{(p^{r_{n-2}+r_{n-1}+1}-1)}/{2},\dots,{(p^{r_{n-2}+r_{n-1}+1}-1)}/{2}\}$, where $r_n$ denotes the $p$-adic valuation of $v_n$.
Then we have, for every $n\geq 0$,
\[\alpha_{n}=\frac{v_{n-1}'}{C' v_n}\left(i+\frac{1}{p}ij\right)\qquad \text{with} \quad C'=\frac{(p-1)q}{C}\quad \text{and}\quad v_{n-1}'=\frac{v_{n-1}}{p^{r_{n-1}}}.\]
In particular, the denominator in $\alpha_n$ is always a multiple of either $q$ or $p-1$, so that it cannot be either $1$ or $-1$. As a consequence, the continued fraction expansion of $\alpha_0$ never stops, which proves the claim.
\end{proof}

\begin{corollary}
    If $q$ divides $p-1$, then the element
    \[\alpha_0=\frac{1}{q}\left(i+\frac{1}{p}ij\right)\]
    has purely periodic continued fraction expansion. The period has length $1$ if $p=3$ and $q=2$, and $2$ otherwise.
\end{corollary}
\begin{proof}
    This is just a specialization of the proof of Theorem~\ref{thm:QCFFfail}. Namely, we can write Bézout's identity explicitly as
    \[p+ \frac{1-p}{q}\cdot q=1,\]
    so that, by~\eqref{eqn:alphaexample},
    \[\alpha_1=\frac{1}{p-1}\left(i+\frac{1}{p}ij\right).\]
    Bézout's identity is now
    \[p + (-1)(p-1)=1,\]
    which gives $\alpha_2=\alpha_0$.
\end{proof}

\section{Roots of some quadratic quaternionic polynomials}\label{sect:polinomi}
In this section we show how Theorem~\ref{teo:muForInfiniteCF} can be exploited to study the roots of a family of quadratic polynomials\,\footnote{Here we generalize in an obvious way the notion of polynomial to this non-commutative setting.} with coefficients in $B$.

A link between quadratic equations over quaternion algebras and continued fractions has already been considered by Hamilton~\cite{Hamilton1852} in the special case $B=\alg{-1}{-1}{\mathbb{R}}$. The general case of quadratic equations over $\alg{-1}{-1}{\mathbb{R}}$ has been dealt with in~\cite{Littlewood1930}. Namely, an equation
\begin{equation}
    \label{eqn:quatQuadric}
\sum_{\ell=1}^{n}\alpha_{0,\ell}X\alpha_{1,\ell}X\alpha_{2,\ell} + \sum_{\ell'=1}^{n'} \beta_{0,\ell'}X\beta_{1,\ell'} + \gamma_0=0,
\end{equation}
with $\alpha_{0,\ell},\alpha_{1,\ell},\alpha_{2,\ell},\beta_{0,\ell'},\beta_{1,\ell'},\gamma_0\in B$, can be rewritten in terms of the components of $X$, say $x_0,\dots,x_3$, with respect to the standard generators of $B$, obtaining
\[f_0(x_0,\dots,x_3) + f_1(x_0,\dots,x_3) i + f_2(x_0,\dots,x_3) j + f_3(x_0,\dots,x_3)  ij=0\]
where $f_0,\dots, f_3$ are quadratic polynomials with real coefficients. Thus, $\alpha=t + xi + yj+ zij\in B$ is a solution of~\eqref{eqn:quatQuadric} if and only if $(t,x,y,z)\in \mathbb{R}^4$ is a solution of the polynomial system
\begin{equation*}
    \begin{cases}
       f_0(x_0, \dots, x_3)=0\\
       f_1(x_0, \dots, x_3)=0\\
       f_2(x_0, \dots, x_3)=0\\
       f_3(x_0, \dots, x_3)=0.
    \end{cases}
\end{equation*}
Therefore, \eqref{eqn:quatQuadric} has either no solution, up to $16$ solutions or infinitely many.
\par Littlewood's arguments can be straightforwardly generalized to quaternion algebras over arbitrary fields. However, to the best of our knowledge, not much more is known about quadratic equations with coefficients in a quaternion algebra $B$ over~$\mathbb{Q}$.

\subsection{Roots of $X^2-aX-1$}
 An element $a\in B$ is said to be \emph{integral} if its minimum polynomial over $\QQ$ has integral coefficients; similarly, $a\in B$ is said to be \emph{$p$-integral} if its minimum polynomial over $\QQ$ has coefficients in $\ZZ[\frac 1 p]$. 
 
\begin{lemma}\label{lem:interoinordine} \
\begin{itemize}
    \item[a)] Let $a\in B$ be integral. Then, there is an order $R$ in $B$ such that $a\in R$.
    \item[b)] Let $a\in B$ be $p$-integral. Then there exists an order $R\in B$ such that $a\in R[\frac 1 p]$.
    \item[c)]  Let $a\in B$  be a $p$-integral element such that $|a|_p>1$. Then, there exists a quaternionic type $\tau=(B,R,p,s)$ such that $a\in s(B)$.
\end{itemize}
\end{lemma}
\begin{proof}  If $a\in \ZZ$, then $a$ lies in every order, so we can assume $a\not\in \ZZ$. In this case $\ZZ[a]$ is an order in a quadratic field $K=\QQ(a)\subseteq B$. Then, we can write $a=n+\sqrt{d}$ with $n,d\in\ZZ$, so that it suffices to show that there is an order $R$ in $B$ containing $x=\sqrt{d}$. By the Skolem-Noether theorem \cite[Thm.\,1.2.1]{voight:quaternion}, there exists an element $y\in R$ such that the $yxy^{-1}=-x$. Then, it is immediate to see that $\ZZ[x,y]$ is an order containing $x$, proving the first part.

  Point  $b)$ is an immediate consequence of $a)$ applied to $p^ka$ for a suitable $k\in \NN$. Finally, we deduce part $c)$ by considering an order $R$ such that $a\in R[\frac 1 p]$ and a floor function $s$ such that $s(a+pR_p)=a$.
\end{proof}

Thanks to the previous lemma, we are able to prove a result about the existence of roots of certain quadratic polynomials. 

\begin{proposition}    \label{prop:special_poly}
    Let $a \in B$ be such that
    \begin{itemize}
        \item $|a|_p>1$ for some odd prime $p$,
        \item $a$ is $p$-integral,
        \item $\frac{\theta(N_\infty(a))}{|a|_{p}}<1$.
    \end{itemize}
    Then, the polynomial $f(X)=X^2-aX-1$ has no root in $B$.
\end{proposition}
\begin{proof}
 Lemma~\ref{lem:interoinordine} and
    Proposition~\ref{prop:uniquenessCF} ensure that there exist an order $R$ and a floor function $s$ such that the periodic continued fraction $[\overline{a}]$ is a continued fraction of type $\tau=(B,R,p,s)$. The $p$-adic limit of $[\overline{a}]$, say $\alpha'$, annihilates $f(X)$ and does not lie in $B$ by Theorem~\ref{teo:muForInfiniteCF}. It is immediate to check that the same holds for $\alpha''=-1/\alpha'$.
   In order to conclude, we only need to prove that $f(X)$ has at most two roots. Let $\alpha\in B_p$ be another root; up to replacing $\alpha$ by $\-\alpha^{-1}$, we can assume that $|\alpha|_p=|a|_p>1$ and $|-\alpha^{-1}|_p<1$. It follows that $s(\alpha)=a$, hence $\alpha$ and $\alpha'$ have the same continued fraction expansion with respect to the type $\tau$. Therefore $\alpha=\alpha'$.
\end{proof}
\begin{remark}
 If $a\notin \mathbb{Q}$, proving that $f(X)$ has only two roots is even easier. In fact, if $\alpha$ is a root of $f(X)$, then $\alpha\neq 0$ and $a=\alpha-1/\alpha\in\mathbb{Q}(\alpha)$. Equivalently, each root of $f(X)$ belongs to $\mathbb{Q}(a)$. Since $\mathbb{Q}(a)$ is a field, it contains no more than two roots of $X^2-aX-1$, proving the claim directly.
\end{remark}

\subsection{More general quadratic polynomials}

The previous argument can be exploited to prove a more general version of Proposition~\ref{prop:special_poly} for a larger class of quadratic polynomials.

\begin{theorem}
\label{cor:quadEquationCF}
    Let $B$ be a division quaternion algebra over $\mathbb{Q}$, $p$ an odd prime at which $B$ ramifies, and $a_0,\dots, a_n$ a sequence of elements in $B$ for some $n\geq0$ such that, for each $i \in \{0,\dots,n\}$,
     \begin{itemize}
        \item $|a_i|_p>1$ ,
        \item there exists a maximal order $R\subset B$ such that $a_i\in R_q$ for every prime $q \neq p$,
        \item if $|a_i-a_j|_p<1$, then $a_i=a_j$.
        \item $\frac{\theta(N_\infty(a_i))}{|a_i|_{p}}<1$.
    \end{itemize}
    Define the sequences $A_0,\dots,A_n$ and $B_0,\dots,B_n$ as in Section~\ref{sec:contFrac}. Then, the polynomial
    \begin{equation}
    \label{eqn:deg2Pol}
        XB_nX + XB_{n-1} - A_n X - A_{n-1}
    \end{equation}
    has no root in $B$.
\end{theorem}
\begin{proof}
 Let $\alpha$ be a root of~\eqref{eqn:deg2Pol}. Equivalently, by~\eqref{eqn:alpha0Qn},
 \begin{equation}
 \label{eqn:periodicCF}
     \alpha=a_0 + \cfrac{1}{a_1+ \cfrac{1}{a_2+\cfrac{1}{\ddots + \cfrac{1}{a_n + \cfrac{1}{\alpha}}}}}.
 \end{equation}
 Proposition~\ref{prop:uniquenessCF} ensures that there exists a floor function $s$ such that $[\overline{a_0, \dots, a_n}]$ is a continued fraction of type $\tau=(B,R,p,s)$. The $p$-adic limit of $[\overline{a_0, \dots, a_n}]$, say $\alpha'$, satisfies~\eqref{eqn:periodicCF} and does not lie in $B$ by Theorem~\ref{teo:muForInfiniteCF}.\\
 Furthermore, following a well-known result for the classical case~\cite{Galois1828}, we remark that~\eqref{eqn:periodicCF} can be rewritten as follows:
\begin{equation*}
(\alpha-a_0)^{-1}=a_1+\cfrac{1}{a_2+\cfrac{1}{\ddots + \cfrac{1}{a_n + \cfrac{1}{\alpha}}}},
\end{equation*}
which gives
\begin{equation*}
((\alpha-a_0)^{-1}-a_1)^{-1}=a_2+\cfrac{1}{\ddots +\cfrac{1}{a_n + \cfrac{1}{\alpha}}},
\end{equation*}
and so by induction 
\begin{equation*}
a_n + \cfrac{1}{a_{n-1}+ \cfrac{1}{a_{n-2}+\cfrac{1}{\ddots + \cfrac{1}{a_0 + \cfrac{1}{-1/\alpha}}}}}=-\frac{1}{\alpha}.
\end{equation*}

 Therefore, another root of~\eqref{eqn:deg2Pol}, say $\alpha''$, is the inverse of the opposite of the $p$-adic limit of $[\overline{a_n, \dots, a_0}]$, which is a continued fraction of type $\tau$ by Proposition~\ref{prop:uniquenessCF}. Moreover, $\alpha''$ does not lie in $B$ by Theorem~\ref{teo:muForInfiniteCF}.
 
 In order to conclude, we need to show that~\eqref{eqn:deg2Pol} has no root in $B_p$ other than $\alpha'$ and $\alpha''$. Define $A_{n+1},A_{n+2},\dots$ and $B_{n+1},B_{n+2},\dots$ the sequences associated to $[\overline{a_0, \dots, a_n}]$. Since $a_{n+1}=a_0$, the recursive formulas defining $A_i$ and $B_i$ yield
 \begin{align}
    \label{eqn:AiBiPeriodicRecurrence}
    A_{k(n+1)+\ell}&=A_{k(n+1)+\ell-1}a_\ell + A_{k(n+1)+\ell-2},\\ B_{k(n+1)+\ell}&=B_{k(n+1)+\ell-1}a_\ell + B_{k(n+1)+\ell-2}, \nonumber
\end{align}
    for every $k\geq 1$ and $\ell \in \{0,\dots, n\}$, and $\alpha$ is a root of~\eqref{eqn:deg2Pol} if and only if
    \begin{align}
    \nonumber
        \alpha B_{k(n+1)+n}\alpha+ \alpha B_{k(n+1)+n-1}&=A_{k(n+1)+n} \alpha +A_{k(n+1)+n-1}\\ \label{eqn:alphaEqForDeg2}
        \alpha B_{k(n+1)+n}\left(\alpha+  B_{k(n+1)+n}^{-1} B_{k(n+1)+n-1}\right)&= A_{k(n+1)+n} \left(\alpha +A_{k(n+1)+n}^{-1}A_{k(n+1)+n-1}\right)
    \end{align}
    for every $k\geq0$. We remark that~\eqref{eqn:AiBiPeriodicRecurrence} allows rewriting $A_{k(n+1)+n}^{-1}A_{k(n+1)+n-1}$ as follows:
    \begin{align*}
        A_{k(n+1)+n}^{-1}A_{k(n+1)+n-1}&=\frac{1}{A_{k(n+1)+n-1}^{-1}A_{k(n+1)+n}}\\
        &=\frac{1}{a_n+ A_{k(n+1)+n-1}^{-1}A_{k(n+1)+n-2}}\\
        &=\cfrac{1}{a_n+ \cfrac{1}{a_{n-1} + A_{k(n+1)+n-2}^{-1}A_{k(n+1)+n-3}}}\\
        &=\cfrac{1}{a_n+ \cfrac{1}{a_{n-1} + \cfrac{1}{\ddots + \cfrac{1}{a_0 + A_{(k-1)(n+1)+n}^{-1}A_{(k-1)(n+1)+n-1}}}}}\\
        &=\cfrac{1}{a_n+ \cfrac{1}{a_{n-1} + \cfrac{1}{\ddots+ \cfrac{1}{a_0}}}}\\
        &=\frac{1}{\tilde{Q}_{k(n+1)}}
    \end{align*}
    for every $k\geq1$, where $\{\tilde{Q}_i\}_{i \in \mathbb{N}}$ is the sequence of convergents of $[\overline{a_n,\dots,a_0}]$. The same argument yields $B_{k(n+1)+n}^{-1}B_{k(n+1)+n-1}=1/\tilde{Q}_{k(n+1)-1}$. Denote by $\{Q_i\}_{i \in \mathbb{N}}$ the convergents of $[\overline{a_0,\dots,a_n}]$. By Proposition~\ref{prop:CFconvergence}, for any $\epsilon>0$ there exists $k_\epsilon \in \NN$ such that
    \begin{align*}
        \max\left\{\left|Q_{k_\epsilon(n+1)+n}-\alpha'\right|_p,\left|\alpha''+\frac{1}{\tilde{Q}_{k_\epsilon(n+1)-1}}\right|_p, \left|\alpha''+\frac{1}{\tilde{Q}_{k_\epsilon(n+1)}}\right|_p\right\}<\epsilon.
    \end{align*}
    We ease the notation by setting $A_\epsilon=A_{k_\epsilon(n+1)+n}$ and $B_\epsilon=B_{k_\epsilon(n+1)+n}$, and rewrite~\eqref{eqn:alphaEqForDeg2} as
    \begin{align*}
    \alpha B_\epsilon\left(\alpha+  \frac{1}{\tilde{Q}_{k_\epsilon(n+1)-1}}\right)&= A_\epsilon \left(\alpha+ \frac{1}{\tilde{Q}_{k_\epsilon(n+1)}}\right)\\
        (\alpha B_\epsilon - A_\epsilon)\left(\alpha -{\alpha''}\right)&= -\alpha B_\epsilon\left(\alpha''+\frac{1}{\tilde{Q}_{k_\epsilon(n+1)-1}}\right)+ A_\epsilon\left(\alpha''+\frac{1}{\tilde{Q}_{k_\epsilon(n+1)}}\right).
    \end{align*}
    Thus, if $\alpha \neq \alpha''$ the following estimates hold:
    \begin{align*}
    \left|\alpha-A_\epsilon B_\epsilon^{-1} \right|_p&\leq \epsilon\cdot \frac{1}{ \left|\alpha-\alpha''\right|_p} \cdot \max\left\{|\alpha|_p, \left|A_\epsilon B_\epsilon^{-1}\right|_p\right\}\\
        \left|\alpha-A_\epsilon {B_\epsilon}^{-1} \right|_p&\leq \epsilon\cdot \frac{1}{ \left|\alpha-{\alpha''}\right|_p} \cdot \max\left\{|\alpha|_p, |\alpha'|_p,\epsilon \right\}.
    \end{align*}
    This proves that $\alpha$ converges $p$-adically to $\alpha'$, i.e.\ $\alpha=\alpha'$.
\end{proof}

\printbibliography
\end{document}